\documentclass{amsart}

\usepackage[latin1]{inputenc}
\usepackage{amsmath,amssymb}
\usepackage{color}

\usepackage{enumerate}
\usepackage[bookmarks, citecolor = blue]{hyperref}
\usepackage{algorithm}
\usepackage{algorithmic}

\newcommand{\T}{\mathbb{T}(\ell)}
\newcommand{\sat}{{\textnormal{sat}}}
\newcommand{\reg}{\textnormal{reg}}
\newcommand{\Sk}{S^{(\ell,n)}} 
\newcommand{\PP}{\mathcal P}

\numberwithin{equation}{section}

\newtheorem{theorem}{Theorem}[section]
\newtheorem{lemma}[theorem]{Lemma} 
\newtheorem{corollary}[theorem]{Corollary} 
\newtheorem{proposition}[theorem]{Proposition} 

\theoremstyle{definition}
\newtheorem{definition}[theorem]{Definition} 
\newtheorem{example}[theorem]{Example} 

\theoremstyle{remark}
\newtheorem{remark}[theorem]{Remark}

\DeclareMathAlphabet{\mathpzc}{OT1}{pzc}{m}{it}

\begin{document}

\title[Quasi-stable and Borel-fixed ideals]{Quasi-stable ideals and Borel-fixed ideals with a given Hilbert Polynomial} 
\author{Cristina Bertone}

\address{Dipartimento di Matematica \lq\lq G. Peano\rq\rq, via Carlo Alberto 10, Torino, Italy}
\email{cristina.bertone@unito.it}

\maketitle

\begin{abstract} 
The present  paper  investigates properties of quasi-stable ideals and of Borel-fixed ideals in a polynomial ring $k[x_0,\dots,x_n]$, in order to design two algorithms: the first one   takes as  input $n$ and an admissible {Hilbert polynomial} $P(z)$, and outputs the complete list of saturated quasi-stable ideals in the chosen polynomial ring with the given Hilbert polynomial. The second algorithm has an extra input, the characteristic of the field $k$, and outputs the complete list of saturated Borel-fixed ideals in $k[x_0,\dots,x_n]$ with Hilbert polynomial $P(z)$. The key tool for the proof of both algorithms is the combinatorial structure of a quasi-stable ideal,  in particular we use a special  set of generators for the considered ideals, the  {Pommaret basis}.
\end{abstract}

\section{Introduction}
In the present paper, we are interested in \emph{quasi-stable} monomial ideals in a polynomial ring $k[x_0,\dots,x_n]$. These ideals have a very rich combinatorial structure (summarized in Theorem \ref{thm:BorelEquiv}) and appear in literature under several different names: for instance, they are called \emph{Borel type ideals} in \cite{HPV03} or \emph{monomial ideals of nested type} in \cite{BG06} or \emph{weakly stable ideals} \cite{CS05}. We are interested in these special type of monomial ideals because among them we can find the \emph{Borel-fixed} ones: these are the ideals which are stable under the action of the Borel subgroup of invertible matrices. Borel-fixedness is a property depending on the characteristic of the field $k$ (Example \ref{ex:BorelChar}, (\ref{BorelChar_i})), while quasi-stability is independent of the characteristic. Borel-fixed ideals are well-known and used in algebra and also in algebraic geometry because given an ideal $I$ and a term order, the \emph{generic initial ideal of $I$} is Borel-fixed. 

Further, under the hypotheses that $k$ is infinite and has characteristic 0, in some recent papers \cite{CR},\cite{BCLR}, \cite{BLR} the properties of Borel-fixed ideals were exploited to bring on a computational and effective study of the \emph{Hilbert scheme}, the scheme parameterizing flat families of saturated ideals in $k[x_0,\dots,x_n]$ having the same Hilbert polynomial $P(z)$. This new direct approach led to a better insight on  Hilbert schemes: see for instance \cite{BCR}, \cite{FR}. In these two papers a key tool to explicit study specific Hilbert schemes is the computation of  the list of Borel-fixed ideals in $k[x_0,\dots,x_n]$ with a given Hilbert polynomial.  Under the hypothesis that the characteristic of $k$ is 0, this list can be obtained by the algorithm presented in \cite{CLMR11} and later improved in \cite{L} or by the algorithm presented in \cite{MN}.

In \cite{CMR13}, the authors generalize some of the computational techniques developed in \cite{CR}, \cite{BCLR} to the case of quasi-stable ideals. Hence, there is a missing step in order to generalize  the open cover of the Hilbert scheme defined in \cite{BLR} and apply it  also in the case of a field $k$ of arbitrary cha\-rac\-te\-ri\-stic: the explicit computation of the complete list of saturated  quasi-stable ideals (and in particular of Borel-fixed ideals)   in the polynomial ring $k[x_0,\dots,x_n]$ with a given Hilbert polynomial. 

The goal of the present  paper is filling this gap,   designing two algorithms: they both  take as an input the number of variables of the polynomial ring and an admissible Hilbert polynomial. The output for one algorithm is the complete list of saturated quasi-stable ideals in the chosen polynomial ring with the target Hilbert polynomial, while the output for the other one is  the complete list of saturated  Borel-fixed ideals with the same features, taking in input also the characteristic of the coefficient field (Algorithms  \ref{alg:QSideals} and \ref{alg:PBorel}).

For our goal, the most useful property of a quasi-stable ideal is that it has a \emph{Pommaret basis}, a special set of generators which can be easily computed from its minimal monomial basis. A Pommaret basis has got many useful features (see \S \ref{sec:Pommaret}) that allows us to adapt the arguments of \cite{CLMR11} for  Borel-fixed ideals in characteristic 0 to both quasi-stable (\S \ref{sec:AlgorithmQS}) and Borel-fixed  ideals in arbitrary characteristic (\S \ref{sec:AlgoBorel}). 

The algorithms we obtain can be easily implemented on any Computer Algebra System. We wrote a non-optimized implementation in Maple \cite{Maple}, that could compute the complete list of quasi-stable and Borel-fixed ideals in several non-trivial cases. We briefly comment two of these computations in \S \ref{sec:examples}.

\section{Notations and generalities}

For every $0\leq \ell< n$, we consider the variables $x_\ell,\dots,x_n$, ordered as $x_\ell<\cdots < x_{n-1} < x_n$ (see \cite{Seiler2009I,Seiler2009II}). This is a non-standard way to sort the variables, but it is suitable for our purposes. In some of the papers we refer to, variables are ordered in the opposite way, hence the interested reader should pay attention to this when browsing a reference.

A \emph{term} is a power product $x^\alpha= x_\ell^ {\alpha_\ell}\cdots x_n^{\alpha_n}$. We denote by $\mathbb T(\ell)$ the set of terms in the variables $x_\ell,\dots,x_n$.
We denote by $\max(x^\alpha)$ the biggest variable that appears with non-zero exponent  in $x^\alpha$ and, analogously, $\min(x^\alpha)$ is the smallest variable
that appears with non-zero exponent in $x^\alpha$. The \emph{degree} of a term is $\deg(x^\alpha) =\sum_{i=\ell}^n \alpha_i =\vert\alpha\vert$.

Let $k$ be a  field and consider the polynomial ring $\Sk:=k[x_\ell,\dots,x_n]$ with the standard graduation. We write $\Sk_t$ for the set of homogeneous polynomials of degree $t$ in $\Sk$. Since $\Sk=\oplus_{i\geq 0} \Sk_i$, we define $\Sk_{\geq t}:=\oplus_{i\geq t} \Sk_i$. The ideals we consider in $\Sk$ are always homogeneous. If $I\subset \Sk$ is a homogeneous ideal, we write $I_t$ for $I\cap \Sk_t$ and $I_{\geq t}$ for $I\cap \Sk_{\geq t}$. The ideal $I_{\geq t}$ is the \emph{truncation} of $I$ in degree $t$.

The ideal $J\subseteq \Sk$ is \emph{monomial} if it is generated  by a set of terms. If $J$ is a monomial ideal, a minimal set of terms generating it is unique and we call it \emph{the monomial basis} of $J$ and denote it by $B_J$.

If $I\subset \Sk$ is a homogeneous ideal, we define the \emph{Hilbert function} of $\Sk/I$ as $h_{\Sk/I}(t)=\dim_k(\Sk_t/I_t)$. If we consider a monomial ideal $J$, observe that $h_{\Sk/J}(t)=\vert \mathbb T(\ell)\vert-\vert J_t\cap \mathbb T(\ell)\vert$. For every homogeneous ideal $I\subset \Sk$ there is a nu\-me\-ri\-cal polynomial  $P(z)\in \mathbb Q[z]$ such that for every $t\gg 0$, $h_{\Sk/I}(t)=P(t)$, which is called the \emph{Hilbert polynomial} of $I$.   A polynomial is an \emph{admissible} Hilbert polynomial if it is the Hilbert polynomial of some ideal.

Let $I$ be a homogeneous ideal in $\Sk$.  Consider its graded minimal free resolution
\[
0\rightarrow E_n\rightarrow \cdots \rightarrow E_1\rightarrow E_0\rightarrow I\rightarrow 0,
\]
where $E_i=\oplus_j \Sk(-a_{ij})$.
The ideal $I$ is $m$-regular if $m\geq a_{ij}-i$ for every $i,j$. The \emph{Castelnuovo-Mumford regularity} of $I$, or simply the \emph{regularity} of $I$, denoted by $\reg(I)$, is the smallest $m$ for which $I$ is $m$-regular (see for instance \cite{Eis2}). Using the above notations for Hilbert function and polynomia, we recall that for every $s\geq \reg(I)$, $h_{\Sk/I}(s)=P(s)$.

If $I$ and $J$ are homogenous ideals in $\Sk$, we define $(I:J)$ as the ideal $\left\{ f\in \Sk\,\vert\, fJ\subseteq I \right\}$;  we will write $(I:x_i)$ for $(I:(x_i))$. Further, we define $(I:J^\infty):=\cup_{j\geq 0}(I:J^j)$; again, we will write $(I:x_i^\infty)$ for $(I:(x_i)^\infty)$.

The ideal $I\subset \Sk$ is \emph{saturated} if $I=(I:(x_\ell,\dots,x_n)^\infty)$. The \emph{saturation} of $I$ is $I^\sat=(I:(x_\ell,\dots,x_n)^\infty)$  and $I$ is $m$-saturated if $I_t=I_t^\sat$ for every $t\geq m$.


\section{Quasi-stability and stability}\label{Sec:QS,S,SS}
In the present section we  introduce \emph{quasi-stability} and \emph{stability} of a monomial ideal. Both properties do not depend on the characteristic of the field we work on.  A thorough reference on this subject is \cite{Seiler2009II}.

 As already mentioned, quasi-stable ideals  can be found in literature under several different names, see for instance \cite{BG06,CS05,HPV03}. Here we consider the definition given in \cite{CMR13}, which better fits our purposes.

\begin{definition}\cite[Definition 4.4]{CMR13}\label{def:qstable}
Let $J\subset \Sk$ be a monomial ideal. 
\begin{enumerate}[(i)]
\item\label{qstable_i} $J$ is \emph{quasi-stable} if for every $x^\alpha \in J\cap \mathbb T(\ell)$, for every $x_j>\min(x^\alpha)$, there is $s\geq 0$ such that $\displaystyle\frac{x_j^s x^\alpha}{\min(x^\alpha)}\in J$.
\item $J$ is \emph{stable} if for every  $x^\alpha \in J\cap \mathbb T(\ell)$, for every $x_j>\min(x^\alpha)$ then  $\displaystyle\frac{x_jx^\alpha}{\min(x^\alpha)}\in J$.
\end{enumerate}

\end{definition}

\begin{remark}
In order to establish if $J$ is  quasi-stable or stable 
it is enough to check the conditions of Definition \ref{def:qstable} on the terms  $x^\alpha \in B_J$.
\end{remark}

\begin{example}
We consider the polynomial ring $S^{(0,2)}$ and the monomial ideal $J=(x_1,x_2^2)\subset S^{(0,2)}$. In order to prove that $J$ is quasi-stable, we only need to check that for some $s\geq 0$, $\displaystyle\frac{x_2^s x_1}{\min(x_1)}\in J$. It is sufficient to take $s=2$, hence $J$ is quasi-stable. On the other hand, $J$ is not stable because $\displaystyle\frac{x_2x_1}{\min(x_1)}=x_2\notin J$.
\end{example}

Combining the various definitions and properties appearing in literature, we get that there are several equivalent properties characterizing quasi-stable ideals.

\begin{theorem}\label{thm:BorelEquiv}
Let $J\subset \Sk$ be a monomial ideal, $P(z)$ be the Hilbert polynomial of $J$ and $d$ be the degree of $P(z)$. The following conditions are equivalent:
\begin{enumerate}[(i)]
\item \label{BorelEquiv_i}$J$ is quasi-stable;
\item \label{BorelEquiv_ii}for each term $x^\alpha\in J$ and for all integers $i,j,m$ such that $\ell\leq i<j\leq n$ and $x_i^m$ divides $x^\alpha$, there exists $s\geq 0$ such that $\displaystyle\frac{x_j^s x^\alpha}{x_i^m}\in J$;
\item \label{BorelEquiv_iii}for each term $x^\alpha\in J$ and for all integers $i,j$ such that $\ell\leq i<j\leq n$, there exists $s\geq 0$ such that $\displaystyle\frac{x_j^s x^\alpha}{x_i^{\alpha_i}}\in J$;
\item \label{BorelEquiv_v}$(J:x_i^\infty)=(J:(x_n,\dots,x_i)^\infty),\, \forall i=\ell,\dots,n$;
\item  \label{BorelEquiv_vi}the variable $x_\ell$ is not a zero divisor for $S/J^{\mathrm{sat}}$ and for all $\ell\leq j<d$ the
variable $x_{j+1}$ is not a zero divisor for $S/(J,x_\ell,\dots,x_j)^{\mathrm{sat}}$. 
\end{enumerate}
\end{theorem}
\begin{proof}
The equivalences among items \eqref{BorelEquiv_ii}, \eqref{BorelEquiv_iii} and \eqref{BorelEquiv_v} are proved in  \cite[Proposition 2.2]{HPV03}. The equivalence between items \eqref{BorelEquiv_v} and \eqref{BorelEquiv_vi} is proved in \cite[Proposition 3.2]{BG06}.
Although well-known, we prove here the equivalence between items \eqref{BorelEquiv_i} and  \eqref{BorelEquiv_ii}.

Assume that (\ref{BorelEquiv_ii}) holds and apply it for every $x^\alpha$ in the monomial ideal $J$, with $x_i=\min(x^\alpha)$ and $m=1$. Then $J$ fullfills Definition \ref{def:qstable}.

Conversely, assume that $J$ is quasi-stable. First, we prove that for every $x^\alpha$ in $J$, for every $x_i$ dividing $x^\alpha$, for every $x_j>x_i$ there is $s$ such that $\displaystyle\frac{x_j^s x^\alpha}{x_i}\in J$. 
 If $x_i=\min(x^\alpha)$, by  Definition \ref{def:qstable} it is immediate that there is $s\geq 0$ such that $\displaystyle\frac{x_j^s x^\alpha}{x_i}\in J$. 

Consider then $x_i>\min(x^\alpha)$, define $a:=\sum_{\omega>i}\alpha_\omega$ and consider the following terms, which belong to $J$ by (\ref{BorelEquiv_i}), for some $s_t\geq 0$:
\[x^{\alpha(0)}:=x^\alpha, \quad x^{\alpha(t)}=\frac{x_j^{s_t}x^{\alpha(t-1)}}{\min(x^{\alpha(t-1)})},  \quad t=1,\dots,a.\]
Then $\min(x^{\alpha(a)})=x_i$, and then there is $s\geq 0$ such that $\displaystyle\frac{x_j^s x^{\alpha(a)}}{x_i}$ belongs to $J$. Following backwards the definition of the terms $x^{\alpha(t)}$, we obtain that $\displaystyle\frac{x_j^{\overline s} x^\alpha}{x_i} \in J$, where $\overline s=s+\sum_t s_t\geq 0$.

Iterating the above arguments, we obtain condition (\ref{BorelEquiv_ii}) for every $m\geq 0$ such that $x_i^m$ divides $x^\alpha$.
\end{proof}

\begin{remark}
Products, intersections, sums and quotients of quasi-stable ideals are quasi-stable (see \cite[Lemma 4.6]{Seiler2009II}). In particular, if $J\subset \Sk$ is quasi-stable, then $J_{\geq m}$ is quasi-stable for every $m$.
\end{remark}

\section{Pommaret basis of  a monomial ideal}\label{sec:Pommaret}

We now recall the definition and some properties of the Pommaret basis of  a monomial ideal. 
Several of the following definitions and properties hold in a more general setting, that is for \emph{involutive divisions}. For a deeper insight in this  topic, we refer to \cite{Seiler2009I}, \cite{Seiler2009II} and the references therein. For a set of terms $M\subset \mathbb T(\ell)$, $(M)$ is the ideal generated by $M$ in the polynomial ring $\Sk$.

\begin{definition}\cite{Seiler2009I}
Consider $x^\alpha \in \T$. 
The \emph{Pommaret cone} of $x^\alpha \in \T$ is the set of terms $\mathcal C_\PP(x^\alpha)=\lbrace x^\alpha x^\delta\vert \max(x^\delta)\leq \min(x^\alpha)\rbrace$.
Let $M\subset \mathbb T(\ell)$ be a finite set. The \emph{Pommaret span} of $M$ is 
\begin{equation}\label{eq:PCones}
\left< M\right >_\PP:=\bigcup_{x^\alpha \in M} \mathcal C_\PP(x^\alpha).
\end{equation}

The finite set of terms $M$ is a \emph{weak Pommaret basis} if $\left <M\right>_\PP=(M)$ and it is a $\emph{Pommaret basis}$ if the union on the right hand side of \eqref{eq:PCones} is disjoint. 

Let $\overline M\supseteq M$ be a finite set in $\T$ such that $\left< \overline M\right>_\PP=(M)$. We call $\overline M$ a \emph{Pommaret completion} of $M$. An \emph{obstruction} is any element of $(M)\setminus \left<M\right>_\PP$.
\end{definition}

\begin{remark}\label{rem:varieP}\ 

\begin{enumerate}[(i)]
 \item\label{varieP_i} By definition, a (weak) Pommaret basis $M$ contains a finite number of terms. It is not true that every monomial ideal has such a basis \cite[page 231, after the proof of Proposition 6.8]{Seiler2009I}. For instance, the ideal $J=(x_0x_1) \subset S^{(0,1)}$ does not contain a finite set of terms which is a Pommaret basis  \cite[Remark 4.3 i.]{CMR13}. 
\item\label{varieP_ii} If $M$ is a weak Pommaret basis, it is always possible to find $M'\subseteq M$ such that $M'$ is a Pommaret basis \cite[Proposition 2.8]{Seiler2009I}.
\item \label{varieP_iii}A Pommaret basis  is unique and minimal in the following sense: if $M$ is a Pommaret basis and  $M'\subset \T$ is another finite set such that $(M)=\langle M'\rangle_\PP$ then $M\subseteq M'$ \cite[Proposition 2.8, Proposition 2.11]{Seiler2009I}.
\item \label{varieP_iv} 
If $M$ is a Pommaret basis for the ideal $(M)$, then  $M$ obviously contains the minimal monomial basis of $(M)$.
\end{enumerate}
\end{remark}

As already pointed out, it is not possible to find for every monomial ideal $J$ a \emph{finite} subset of $\T$  which is its Pommaret basis. Nevertheless,  quasi-stable ideals are exactly those having a Pommaret basis.

\begin{theorem}\label{qsPB}\cite[Proposition 4.4]{Seiler2009II}\cite[Remark 2.10]{Mall1998} \ 
Let $J$ be a monomial ideal in $\Sk$.
The ideal $J$ is quasi-stable if and only if  $J$ has a Pommaret basis.
Furthermore, the ideal $J$ is stable if and only if its minimal monomial basis is the Pommaret basis.
\end{theorem}

If $M$ is a finite set of terms generating a quasi-stable ideal, then it is possible to compute a Pommaret completion $\overline M$ algorithmically: it is enough to add to the set $M$ the obstructions of kind $x^\alpha x_i$ with $x^\alpha \in M$ and $x_i>\min(x^\alpha)$, obtaining in this way $\overline M$ \cite[Algorithm 2]{Seiler2009I}.

\begin{algorithm}[h]
\caption{\label{alg:completionP} \textsc{Completion}(M)}
\begin{algorithmic}[1]
\REQUIRE $M\subset \T$ finite set of terms generating a quasi-stable ideal
\ENSURE $\overline M$ Pommaret completion of $M$
\STATE $\overline M \leftarrow M$;
\LOOP
\STATE $F\leftarrow \lbrace x^\alpha x_j:x^\alpha\in \overline M,x_j>\min(x^\alpha), x^\alpha x_j \notin \left<\overline M\right >_\PP\rbrace$;
\IF{$F=\emptyset$}
\RETURN $\overline M$;
\ELSE
\STATE choose $x^\beta\in F$ such that no other term in $F$ divides it;
\STATE $\overline M\leftarrow \overline M\cup \lbrace x^\beta\rbrace$;
\ENDIF
\ENDLOOP
\end{algorithmic}
\end{algorithm}

The termination of Algorithm \ref{alg:completionP}  is proved in \cite{GB}. Observe that the output of Algorithm \ref{alg:completionP}  is in general a weak Pommaret basis, from which we can always obtain a Pommaret basis(see Remark \ref{rem:varieP}, (\ref{varieP_ii})).

\medskip
We now state several properties of Pommaret bases and of the ideals they generate. If no proof is given, we give a precise reference for the interested reader. If $J$ is a quasi-stable ideal, we denote its Pommaret basis by $\mathcal P(J)$.

\begin{definition}
Let  $J$ be a quasi-stable ideal in $\Sk$  and let  $\PP(J)$ be its  Pommaret basis. We define the following sets of monomials
\[
\PP(J)(j):=\lbrace x^\alpha\in \PP(J)\vert \min(x^\alpha)=j \rbrace;
\]
\[
\overline{\PP(J)(j)}:=\left\lbrace \displaystyle\frac{x^\alpha}{x_j^{\alpha_j}}\vert x^\alpha \in \PP(J)(j) \right\rbrace.
\]
\end{definition}

\begin{lemma}\label{thm:PBSat}
Let  $J$ be a quasi-stable ideal in $\Sk$  with Pommaret basis  $\PP(J)$ and consider $\ell\leq j\leq n$.
\begin{enumerate}[(i)]
\item The ideal  $(J:(x_n,\dots, x_{j})^\infty)$ has weak Pommaret basis $$\overline{\PP(J)(j)}\cup \bigcup_{i=j+1}^n \PP(J)(i);$$
\item no term in the Pommaret basis of  $(J:(x_n,\dots, x_{j})^\infty)$  is divisible by $x_m$ with $m\leq j$;
\item\label{PBSat_iii} if $J$ is a saturated ideal, then no term in $\PP(J)$ is divisible by $x_\ell$.
\end{enumerate}
\end{lemma}
\begin{proof}
It is sufficient to consider the equivalent property (\ref{BorelEquiv_v}) of Theorem \ref{thm:BorelEquiv} and use \cite[Lemma 4.11]{Seiler2009II}.
\end{proof}

\begin{lemma}\label{lem:1morev}
Let  $M\subset \T$, $\ell\geq 1$, be the Pommaret basis of the ideal $(M)\subset \Sk$. Then $M$ is also the Pommaret basis for the ideal $(M)S^{(\ell-1,n)}$.
\end{lemma}
\begin{proof}
We denote by $\mathcal M$ the ideal generated by $M$ in $S^{(\ell-1,n)}$. Let $x^\alpha$ be a term in $\T$, we denote by $\mathcal C_\PP(x^\alpha)$ the Pommaret cone of $x^\alpha$ in $\Sk$ and by $\mathcal C_\PP^{(\ell-1)}(x^\alpha)$ the Pommaret  cone of $x^\alpha$ in $S^{(\ell-1,n)}$. It is immediate that $\mathcal C_\PP(x^\alpha)\subset \mathcal C_\PP^{(\ell-1)}(x^\alpha)$.

We prove that for every $x^\beta \in \mathcal M$, there is $x^\alpha \in M$ such that $x^\beta\in \mathcal C_\PP^{(\ell-1)}(x^\alpha)$.
If $\min(x^\beta)\geq x_\ell$, then $x^\beta$ is a term in $(M)\subset \Sk$, hence there is $x^\alpha \in M$ such that $x^\beta\in \mathcal C_\PP(x^\alpha)\subset \mathcal C_\PP^{(\ell-1)}(x^\alpha)$.
 Otherwise, we have that $\min(x^\beta)=x_{\ell-1}$. Since no term in $M$ is divisible by $x_{\ell-1}$, $\displaystyle\frac{x^\beta}{x_{\ell-1}^{\beta_{\ell-1}}}$ still belongs to $\mathcal M$ and we are in the previous case: there is $x^\alpha \in M$ such that  $\displaystyle\frac{x^\beta}{x_{\ell-1}^{\beta_{\ell-1}}}\in \mathcal C_\PP(x^\alpha)$ for some $x^\alpha \in M$. More precisely, we can write $\displaystyle\frac{x^\beta}{x_{\ell-1}^{\beta_{\ell-1}}}=x^\alpha x^\delta$ with $\max(x^\delta)\leq \min(x^\alpha)$. But then $x^\beta=x^\alpha(x^\delta x_{\ell-1}^{\beta_{\ell-1}})\in \mathcal C_\PP^{(\ell-1)}(x^\alpha)$.
\end{proof}

\begin{lemma}\label{potenze}  
Let  $J$ be a quasi-stable ideal in $\Sk$  and let  $\PP(J)$ be its  Pommaret basis. Then:
\begin{enumerate}[(i)]
\item \label{potenze_i}$x^\alpha \in J\setminus \PP(J)\ \Rightarrow\ \dfrac{x^\alpha}{\min(x^\alpha)} \in J$;  
\item \label{potenze_ii}$x^\alpha \notin J$ and  $x_ix^\alpha \in J \ \Rightarrow$ either $x_ix^\alpha \in \PP(J)$ or $x_i > \min(x^\alpha)$.
\end{enumerate}
\end{lemma}

\begin{proof}\ 
\begin{enumerate}[(i)]
\item We consider $x^\alpha\in J\setminus \PP(J)$. Then $x^\alpha\in \mathcal C_\PP(x^\beta)$ with $x^\beta \in \PP(J)$: there is $x^\delta$ with $\vert\delta\vert\geq 1$ and $\max(x^\delta)\leq \min(x^\beta)$ such that $x^\alpha=x^\beta\cdot x^\delta$. Hence $\min(x^\alpha)=\min(x^\delta)$ and $\displaystyle\frac{x^\alpha}{\min(x^\alpha)}=x^\beta\cdot \left(\displaystyle\frac{x^\delta}{\min(x^\delta)}\right)$ still belongs to $\mathcal C_\PP(x^\beta)\subset J$.
\item If $x_ix^\alpha\in J\setminus \PP(J)$, then item (\ref{potenze_i}) applies and we obtain $x_i>\min(x^\alpha)$.
\end{enumerate}
\end{proof}

\begin{proposition}\cite[Lemma 2.2, Lemma 2.3, Theorem 9.2, Proposition 9.6]{Seiler2009II}\label{prop:varieS}\ 

\begin{enumerate}[(i)]
\item \label{lem:PBTronc2}
Let $M\subset \T$ be a finite set of terms of degree $s$. If for every $x^\alpha \in M$, for every $x_j>\min(x^\alpha)$, $\displaystyle\frac{x_jx^\alpha}{\min(x^\alpha)}\in M$, then $M$ is a Pommaret basis.
\item \label{lem:reg}Let  $J$ be a quasi-stable ideal in $\Sk$  and let  $\PP(J)$ be its  Pommaret basis. Then the regularity of $J$ is $$\max\lbrace \deg(x^\beta)\vert x^\beta \in \PP(J)\rbrace.$$
\item\label{prop:varieS_iii} Let $J\subset \Sk$ be a quasi-stable ideal generated in degrees less than or equal to $s$. The ideal $J$ is $s$-regular if and only if $J_{\geq s}$ is stable.
\item \label{prop:varieS_iv} Let  $J$ be a quasi-stable ideal in $\Sk$  and consider $s\geq \reg(J)$. Then $J_{\geq s}$ is stable and the set of terms $J_s\cap \T$ is its Pommaret basis.
\end{enumerate}
\end{proposition}

We conclude this section giving a  characterization  of elements of the Pommaret basis of a monomial ideal which can be removed, still having a Pommaret basis.

\begin{definition}\label{def:StMin}
Let $J\subset \Sk$ be  a stable monomial ideal. The term $x^\alpha \in J $ is \emph{St-minimal} if $x_jx^\beta/\min(x^\beta)\neq x^\alpha$, for every $x^\beta \in J$ and for every $x_j>\min(x^\beta).$
\end{definition}

\begin{corollary}\label{PBMin}
Let $J$ be a quasi-stable  \  ideal, consider  $s\geq \reg(J)$ and $x^\alpha \in J_s\cap \T=\PP(J_{\geq s})$. The ideal generated by the set $( J_s\cap \T) \setminus\lbrace x^\alpha\rbrace$ is   stable   if and only if $x^\alpha$ is St-minimal for $J_{\geq s}$. 
\end{corollary}

\begin{proof}
We denote by $M$ the set $(J_s\cap \T) \setminus\lbrace x^\alpha\rbrace$.  

If $x^\alpha$ is St-minimal, then Proposition \ref{prop:varieS} (\ref{lem:PBTronc2}) holds, hence $M$ is a Pommaret basis. We now prove that the ideal generated by $M$ is stable. Suppose that $(M)$ is not stable. By Theorem \ref{qsPB}, we assume that there is a term $x^\beta\in M\setminus B_{(M)}$. This means that $x^\beta=x^\alpha\cdot x^\delta$ with $x^\alpha\in B_{(M)}$ and $\vert \delta \vert \geq 1$. In particular $\vert \alpha\vert<s$, but this is not possible since all terms in $M$ have degree $s$, and the same for the terms  in $B_{(M)}$ (see also Remark \ref{rem:varieP}, (\ref{varieP_iv})).

To prove the other implication, we proceed by contraposition. If $x^\alpha$ is not St-minimal, then there is $x^\beta \in M$ and $x_j>\min(x^\beta)$ such that $\displaystyle\frac{x_jx^\beta}{\min(x^\beta)}=x^\alpha$. Observe that the latter equality means that $x_jx^\beta\in \mathcal C_\PP(x^\alpha)$. Hence, $x_jx^\beta$ belongs to the ideal $(M)$ but does not belong to $\langle M\rangle_\PP$, that is $x_jx^\beta$ is an obstruction and $M$ is not a Pommaret basis.

\end{proof}


\section{Quasi-stable ideals and Hilbert Polynomial}\label{sec:AlgorithmQS}

In the present section, we recall some facts and prove some results giving us a strategy to compute the complete list of quasi-stable saturated ideals with a prescribed Hilbert polynomial. These results will also be used to prove the correctness of the  algorithm we design in Section \ref{sec:algo}.  We follow the lines of \cite{CLMR11}, \cite{L}, where the authors work under the hypothesis that $k$ has characteristic 0. We remove the hypothesis on the characteristic, and deal with quasi-stable ideals, thanks to Pommaret bases.

A crucial point in the algorithm is Gotzmann Regularity Theorem, that we now recall. For a proof, we refer to \cite[Appendix B.6]{vas} and\cite[Section 4.3]{BH}.

\begin{theorem}[Gotzmann Regularity Theorem]\label{GRT} 
Let  $I$ be a homogeneous ideal in $\Sk$. Write the Hilbert polynomial $P(z)$ of $I$ in the unique form
\[
P(z)={z+a_1 \choose a_1}+{z+a_2-1  \choose a_2}+\cdots +{z+a_r-(r-1) \choose a_r}
\]
with $a_1\geq a_2\geq \cdots\geq a_r$. Then the saturation $I^\sat$ of $I$ is $r$-regular.
\end{theorem}

\begin{remark}\label{rem:gotz}\ 

\begin{enumerate}
\item We can rephrase a part of the statement of Theorem \ref{GRT}: for a fixed admissible Hilbert polynomial $P(z)$, there is an integer $r$ upper bounding the regularity of every ideal $I\subset S$ having Hilbert polynomial $P(z)$. The integer $r$ only depends on $P(z)$ and is called \emph{Gotzmann number}. The Gotzmann number is a sharp bound for the regularity of ideals having Hilbert polynomial $P(z)$: indeed, the regularity of the Lex-segment ideal with Hilbert polynomial $P(z)$ is  $r$ \cite{Ba1982}.
\item\label{rem:gotz_ii} If $P(z)$ has Gotzmann number $r$, then the Gotzmann number of the Hilbert polynomial $\Delta P(z):=P(z)-P(z-1)$ is $\leq r$. {This is a consequence of the proof of Gotzmann regularity theorem.}
\end{enumerate}
\end{remark}

Given a quasi-stable ideal $J\subset \Sk$, we call \emph{$x_{\ell+1}$-saturation} of $J$ the ideal $\left((J:x_{\ell}^\infty):x_{\ell+1}^\infty\right)$ and denote it by $J_{x_\ell x_{\ell+1}}$. We say that the quasi-stable ideal $J$ is $x_{\ell+1}$-saturated if $J=J_{x_\ell x_{\ell+1}}$.  

 First, we will establish a connection between properties of a quasi-stable ideal $J\subset \Sk$ and its generic hyperplane section. 
This will give us a recursive method to compute quasi-stable ideals with a given Hilbert polynomial. 

\begin{remark}\label{prop:ricors} We now relate the Hilbert polynomial of a quasi-stable ideal $J$ with that of $(J,x_\ell)/(x_\ell)$. This relation is well-known for strongly stable ideals (see for instance \cite[Section 5]{CLMR11}).

In $\Sk$, consider $J$ a quasi-stable ideal with Hilbert polynomial $P(z)$; the term $x_\ell$ is a non-zero divisor in $\Sk/J^{\mathrm{sat}}$ (Theorem \ref{thm:BorelEquiv}, (\ref{BorelEquiv_vi})). The ideal $(J,x_\ell)/(x_\ell)\subset S^{(\ell+1,n)}$  has the same Hilbert polynomial as $(J,x_\ell)\subset \Sk$. Further, $J^\sat$ has the same Hilbert polynomial as $J$, since for $t\gg 0$, $J^\sat_t=J_t$. Hence, we can consider  the short exact sequence
\[
0\longrightarrow \frac{\Sk}{J^\sat}(t-1)\stackrel{\cdot x_\ell}{\longrightarrow}\frac{\Sk}{J^\sat}(t)\longrightarrow \frac{\Sk}{(J^\sat,x_\ell)}(t)\longrightarrow 0,
\]
and we obtain that the Hilbert polynomial of $(J^\sat,x_\ell)$ is  $\Delta P(z)$. This is also the Hilbert polynomial of $(J,x_\ell)$, since $(J^\sat,x_\ell)_t = ((J^\sat)_t,x_\ell)_t = (J_t, x_\ell)_t = (J, x_\ell)_t$, for every $t\geq r$. 
Further, observe that $(J,x_\ell)/(x_\ell)\subset S^{(\ell+1,n)}$ is quasi-stable too and $((J,x_\ell)/(x_\ell))^\sat$  is generated by  $J_{x_\ell x_{\ell+1}}$ in $S^{(l+1,n)}$, by Theorem \ref{thm:BorelEquiv}, (\ref{BorelEquiv_v}).
\end{remark}

In order to compute all saturated quasi-stable ideals in $\Sk$ with a given Hilbert polynomial $P(z)$, we can use a recursion on the number of variables of the polynomial ring. Assume that we have a complete list $E$ of saturated quasi-stable ideals in $S^{(\ell+1,n)}$ generated in degrees $\leq  r$ with Hilbert polynomial $\Delta P(z)$. Then, we will construct every saturated quasi-stable ideal $J$ in $\Sk$ with Hilbert polynomial $P(z)$ such that $((J,x_\ell)/(x_\ell))^\sat$ is among the ideals in the list $E$.

The following Proposition and Lemmas prove the correctness of this recursive strategy and also show how to construct a quasi-stable ideal $J$ in $\Sk$ with Hilbert polynomial $P(z)$, starting from a quasi-stable ideal in $S^{(\ell+1,n)}$ with Hilbert polynomial $\Delta P(z)$.

\begin{proposition}\label{x1sat}
Let $J\subset \Sk$ be a saturated  quasi-stable ideal, let $P(z)$ be its Hilbert polynomial, $r$ be the Gotzmann number of $P(z)$ and consider $s\geq r$. Consider  the ideal $I:=J_{x_\ell x_{\ell+1}}$ and define $q:=\dim_k I_s-\dim_k J_s$. Then 
the Hilbert polynomial of $I$ is $P(z)-q$.
 \end{proposition}
\begin{proof}
We prove that for every $s\geq r$, if $\dim_k I_s-\dim_k J_s=q$, then $\dim_k I_{s+1}-\dim_k J_{s+1}=q$.  

We consider the set of terms in $I_s\setminus J_s$ and denote them by $x^{\beta_1},\dots,x^{\beta_q}$. It is immediate that $x_\ell x^{\beta_1}, \dots,x_\ell x^{\beta_q}$ belong to $I_{s+1}\setminus J_{s+1}$: indeed, by Lemma \ref{potenze}, if the term $x^\beta$ does not belong to $J$, then $x_\ell x^\beta$ is not in $J$, because $J$ is saturated and no term in its Pommaret basis is divisible by $x_\ell$ (Lemma \ref{thm:PBSat},(\ref{PBSat_iii})). Hence $\dim_k I_{s+1}-\dim_k J_{s+1}\geq q$.

In order to prove the other inequality, we proceed by contradiction. Consider $x^\gamma \in I_{s+1}\setminus J_{s+1}$ with $\min(x^\gamma)\geq x_{\ell+1}$. Since $I$ is the $x_{\ell+1}$-saturation of $J$, we can consider the smallest integer $t>0$ such that $x^\gamma x_{\ell+1}^t\in J$. Hence, for some $x^\alpha \in \PP(J)$, $x^\gamma x_{\ell+1}^t\in C_\PP(x^\alpha)$. More explicitely,$x^\gamma x_{\ell+1}^t=x^\alpha \cdot x^\delta$ with $\max(x^\delta)\leq \min(x^\alpha)$. Observe that $\vert\alpha\vert\leq r<s+1=\vert\gamma\vert$ by Proposition \ref{prop:varieS} (\ref{lem:reg}), hence $\vert\delta\vert\geq 1$. Furthermore, $\min(x^\gamma x_{\ell+1}^t)=x_{\ell+1}=\min(x^\delta)$. We can then simplify $x_{\ell+1}$ and obtain $x^\gamma x_{\ell+1}^{t-1}=x^\alpha\cdot\displaystyle\frac{x^\delta}{x_{\ell+1}}	\in J$. But this contradicts the minimality of $t$. 
\end{proof}

\begin{lemma}\label{CostrPom}
Let $J$ be a quasi-stable ideal in $\Sk$, let $P(z)$ be its Hilbert polynomial  and let $r$ be the Gotzmann number of $P(z)$. For an arbitrary $s\geq r$, consider a St-minimal term $x^\beta \in J_{s}$ with $min(x^\beta)=x_\ell$ and let $M\subset \Sk$ be the set of terms $\{J_s\cap \T\} \setminus\lbrace x^\beta\rbrace$. 
Then the ideal generated by $M$ is stable  and its Hilbert polynomial is $P(z)+1$.
\end{lemma}

\begin{proof}
First, we recall   that $J_s\cap \T$ is the Pommaret basis of $J_{\geq s}$ and the set $M$ is the Pommaret basis for the stable monomial ideal  it generates  (by Lemmas  \ref{prop:varieS}, items (\ref{lem:reg}) and (\ref{prop:varieS_iv}), and \ref{PBMin}).

We now show that for every $t\geq 0$, $J_{s+t}\setminus (M)_{s+t}$ contains only the term $x^\beta x_\ell^{t}$.
 We proceed by induction on $t$. If $t=0$, $J_s\setminus (M)_s=J_s\setminus M=\lbrace x^\beta\rbrace$ by definition of $M$. We can assume that the thesis holds for every integer smaller than $t>0$. 

Any term in $J_{s+t}\setminus (M)_{s+t}$ is a multiple of a term in $J_{s+t-1}\setminus (M)_{s+t-1}=\lbrace x^\beta x_\ell^{t-1}\rbrace$, by inductive hypothesis. Hence, consider $x^\gamma=x^\beta x_\ell^{t-1} x_j$ with $x_j>x_\ell$. Since $\min(x^\beta)=x_\ell$, there is $x^\alpha \in J_s$, $x^\alpha\neq x^\beta$, such that $x^\gamma$ belongs to $\mathcal C_\PP(x^\alpha)$. But $x^\alpha$ also belongs to $M$, hence $x^\gamma$ belongs to $(M)$ too.

Suppose now that $x^\beta x_\ell^{t}\in (M)$. Since $M=\{J_s\cap \T\} \setminus\lbrace x^\beta\rbrace$ is a Pommaret basis for $(M)$, there is a term $x^\alpha \in J_{s}$, $x^\alpha\neq x^\beta$, such that $x^\beta x_\ell^{t}\in \mathcal C_\PP(x^\alpha)$. Since $x_\ell=\min(x^\beta x_\ell^t)$ and $t>0$, then $x^\beta x_\ell^{t-1}$ belongs to $\mathcal C_\PP(x^\alpha)\subset (M)$ too, against the inductive hypothesis. Hence, $J_{s+t}\setminus (M)_{s+t}=\lbrace x^\beta x_\ell^t\rbrace$. 

\end{proof}

\begin{lemma}\label{lem:x1sat}
Let $I$ and $J$ be quasi-stable ideals in $\Sk$, let $P_1(z)$ be the Hilbert polynomial of $I$ and $P_2(z)$ be the Hilbert polynomial of $J$. If, for every $s\gg 0$, we have $I_s\subset J_s$ and $P_1(z) = P_2 (z) + a$, with $a\in \mathbb N$, then $I$ and $J$ have the same $x_{\ell+1}$-saturation and for every $s\gg 0$ there is a term $x^\alpha \in J_s\setminus I_s$, with $\min(x^\beta)=x_\ell$, which is St-minimal.
\end{lemma}
\begin{proof} Let $s \geq \max\{\reg(I), \reg(J)\}$. If $a=1$, let $x^\alpha$ be the unique term in $J_{s}\setminus I_{s}$. Then, both $x^\alpha x_{\ell}$ and $x^\alpha x_{\ell+1}$ belong to $J_{s+1}$. Since $J$ and $I$ are quasi-stable and $s \geq \max\{\reg(I), \reg(J)\}$, by Lemma  \ref{prop:varieS},  (\ref{prop:varieS_iv}), $J_{s}\cap \T$ is the Pommaret basis of $J_{\geq s}$ and $I_{s}\cap \T$ is the Pommaret basis of $I_{\geq s}$. Suppose that for some $t>0$, $x^\alpha x_{\ell}^t\in I_{s+t}$: then $x^{\alpha}x_{\ell}^t\in \mathcal C_\PP(x^\beta)$ for some $x^\beta\in I_s\subset J_s$. This is not possible, since $x^\alpha \in J_s$, hence $x^\alpha x_{\ell}^t$ is not in the Pommaret cone of any other term in $J_s$.

Since for every $t>0$, $J_{s+t}\setminus I_{s+t}$ contains the unique term $x^\alpha x_{\ell}^t$, $x^\alpha x_{\ell+1}$ belongs to $I_{s+1}$. This is enough to say that $I$ and $J$ have the same $x_{\ell+1}$-saturation,  by Lemma \ref{thm:PBSat}. 
Further,  $I_s\cap\T=(J_s\cap \T)\setminus\lbrace x^\alpha\rbrace$ and this implies that $x^\alpha$ is St-minimal, by Lemma \ref{PBMin}. Finally, every term $x^\beta\in \mathcal C_\PP(x^\alpha)$, $\vert \beta \vert\geq s$, does not belong to $I_{\geq s}$, because its Pommaret basis is $I_s=J_s\setminus \lbrace x^\alpha\rbrace$. If $\min(x^\alpha)=x_j>x_\ell$, then the number of terms in $\mathcal C_\PP(x^\alpha)$ of any fixed degree $\geq s$ would be strictly bigger than 1, in contradiction with $a=1$. Finally, for every $t>0$, the term $x^\alpha x_\ell^t$ is St-minimal for $J_{\geq s+t}$.

If  $a> 1$, the thesis follows
 by induction and by applying Lemma \ref{CostrPom}: indeed, observe that among the terms in the set $J_s\setminus I_s$, one is St-minimal, otherwise $I_s$ would not be stable (Lemma \ref{PBMin})  and its minimal variable must be $x_\ell$, otherwise $P_1(z)-P_2(z)$ would not be constant. 
\end{proof}

\begin{remark}
Proposition \ref{x1sat},  Lemmas \ref{CostrPom} and \ref{lem:x1sat} generalize \cite[Propositions 4.2, 4.3, 4.4]{CLMR11} to quasi-stable ideals. Indeed,  the quoted results of \cite{CLMR11} are proved under the specific hypothesis that $J$ is a strongly stable ideal (see Remark \ref{combinpbase}).
\end{remark}

\section{An algorithm to compute quasi-stable ideals with a given Hilbert polynomial}\label{sec:algo}

We  now present the algorithm that computes the complete list of quasi-stable monomial ideals $J\subset \Sk$ with a given Hilbert polynomial $P(z)$.
More precisely, the algorithm takes as input $n$, $0\leq \ell\leq n-1$ and an {admissible} Hilbert polynomial $P(z)$ for ideals in $\Sk$, and returns the list of saturated quasi-stable ideals $J$ in $\Sk$ having Hilbert polynomial  $P(z)$. The list we  obtain is independent on the characteristic of the field $k$.

We define an auxiliary function: \\
$\textsc{StMinimal}(J,\ell,n,s)$: takes as input the quasi-stable ideal $J\subset \Sk$ and the integer $s$ and returns the St-minimal elements of $J_{s}$ with minimal variable $x_\ell$.

\begin{algorithm}[h]
\caption{\label{alg:Remove}\textsc{Remove}$(I,\ell, n, s , q)$}
\begin{algorithmic}[1]
\REQUIRE  $I$ set of monomials generating a quasi-stable ideal;
\REQUIRE $\ell$ first index of the variables in the polynomial ring;
\REQUIRE  $n$  least index of the variables of the polynomial ring;
\REQUIRE  $s$ upper bound on $\reg(I)$;
\REQUIRE  $q$ non-negative integer;
\ENSURE $L$ set of the quasi-stable ideals $J$  obtained by removing $q$ St-minimal terms divisible by $x_\ell$ from $I_s$ and then saturating;
\STATE $L\leftarrow \emptyset$;
\IF{$q$=0}
\RETURN\label{step0} $L \cup I^\sat$;
\ELSE
\STATE\label{stepMin} $M\leftarrow \textsc{StMinimal}(I,\ell,n,s)$;
\FORALL{$x^\alpha\in M$}
\STATE\label{stepalgo2ricors} $L\leftarrow L\cup \textsc{Remove}(I_{s}\setminus\lbrace x^\alpha\rbrace,\ell, n,  s , q-1)$;
\ENDFOR
\ENDIF
\RETURN $L$;
\end{algorithmic}
\end{algorithm}

\begin{theorem}\label{thm:corrRem}
Algorithm \ref{alg:Remove}, \textsc{Remove}$(I,\ell, n,  s , q)$, returns the set of all quasi-stable ideals in the polynomial ring $\Sk$ contained in $I_s$, having the same $x_{\ell+1}$-saturation as $I$  
and having  Hilbert polynomial $P_1(z)+q$, where $P_1(z)$ is the Hilbert polynomial of $I$.
\end{theorem}
\begin{proof}
If $q=0$, Algorithm \ref{alg:Remove} terminates at line \ref{step0} and its output is correct. 

If $q>0$, at line \ref{stepMin} the algorithm computes the set of St-minimal terms $x^\alpha$ with $\min(x^\alpha)=x_\ell$. By Lemma \ref{CostrPom}, the set of terms in $I_s\setminus \{x^\alpha\}$ generates a stable ideal with Hilbert polynomial $P_1(z)+1$. 

The algorithm terminates because at each recursive call at line \ref{stepalgo2ricors},  the number of terms to remove decreases.

Furthermore, observe that applying Algorithm \ref{alg:Remove} on the quasi-stable ideal $I$ with Hilbert polynomial $P(z)$, we obtain as output ideals having the same $x_{\ell+1}$-saturation as $I$, by Lemma \ref{lem:x1sat}. 
 \end{proof}

\begin{algorithm}[h]
\begin{algorithmic}[1]
\caption{\label{alg:QSideals}  \textsc{QuasiStable}$(\ell,n,P(z),s)$}
\REQUIRE $\ell$ first index of the variables in the polynomial ring;
\REQUIRE  $n$  least index of the variables of the polynomial ring;
\REQUIRE  $P(z)$ admissible Hilbert polynomial;
\REQUIRE $s$ positive integer upper bounding the Gotzmann number of $P(z)$;
\ENSURE $F$ set of the saturated quasi-stable ideals $J$ in the polynomial ring $\Sk$ having Hilbert polynomial $P(z)$;
\IF{$P(z)=0$}
\RETURN $\lbrace (1)\rbrace$;\label{step:primoReturn}
\ELSE
\STATE$E\leftarrow \textsc{QuasiStable}(\ell+1,n,\Delta P(z),s)$;\label{step:ricors} 
\STATE $F\leftarrow \emptyset$;
\FORALL{$J\in E$}
\STATE\label{step1morevar} $I\leftarrow J\cdot k[x_{\ell},\dots,x_n]$;
\STATE\label{stepdefq} $q\leftarrow P(s)-{n-\ell+s \choose s}+\dim_k I_s$;
\IF{$q\geq 0$}
\STATE\label{steprem}{$F\leftarrow F\cup \textsc{Remove}(I,\ell,n,s,q)$};
\ENDIF
\ENDFOR
\ENDIF
\RETURN $F$;
\end{algorithmic}
\end{algorithm}

\begin{theorem}\label{thm:algoqs}
Algorithm \ref{alg:QSideals}, \textsc{QuasiStable}$(\ell,n,P(z),s)$, returns the set of all quasi-stable sa\-tu\-ra\-ted ideals in the polynomial ring $\Sk$ with Hilbert polynomial $P(z)$.
\end{theorem}

\begin{proof}
We prove correctness of Algorithm  \ref{alg:QSideals} by induction on  $\Delta^m P(z)$. It is sufficient to consider $s=r$, where $r$ is the Gotzmann number of $P(z)$, which upper bounds  the Gotzmann number of $\Delta^m P(z)$ for every $m\geq 0$ (see Remark \ref{rem:gotz}, item \ref{rem:gotz_ii}.).

If $P(z)=0$, then the ideal $(1)$ is the only saturated quasi-stable ideal  in $\Sk$ , hence  Algorithm  \ref{alg:QSideals} returns the correct set (line \ref{step:primoReturn}).

Assume now that $P(z)\neq 0$ and that Algorithm  \ref{alg:QSideals} returns the correct set for $\Delta P(z)$. Then, the recursive call at line \ref{step:ricors}  returns the complete list of the $x_{\ell+1}$-saturations of the ideals we look for, as observed in Remark \ref{prop:ricors}. 
Consider $J\subset S^{(l+1,n)}$ belonging to the output of $ \textsc{QuasiStable}(\ell+1,n,\Delta P(z),r)$ and observe that by Lemma \ref{lem:1morev} and Proposition \ref{prop:varieS} \eqref{prop:varieS_iii}, the ideal $I=J\cdot \Sk$ defined at line \ref{step1morevar} of Algorithm \ref{alg:QSideals} is quasi-stable, hence $I_r$ is stable. Furthermore, the Hilbert polynomial of $I$ is $P(z)+q$, were $q$ is defined at line \ref{stepdefq}. There are three possibilities:
\begin{itemize}
\item if $q<0$, there exist no quasi-stable ideals $I\subseteq \Sk$ with Hilbert polynomial $P(z)$ and such that $(I,x_\ell)/(x_\ell)\simeq J$, hence $J$ has to be discarded, by Proposition \ref{x1sat};
\item if $q=0$ $J^\sat\cdot \Sk$ is one of the ideals sought;
\item if $q>0$, we apply algorithm Algorithm \ref{alg:Remove} to obtain the quasi-stable ideals $I\subseteq \Sk$ with Hilbert polynomial $P(z)$ and such that $(I,x_\ell)/(x_\ell)\simeq J$. 
\end{itemize}
\end{proof}

\begin{remark}
Observe that the integer $s$ input of Algorithm \ref{alg:QSideals} upper bounds the regularity of every ideal $I\subseteq S^{(j,n)}$ arising along the computations (as already pointed out in Remark \ref{rem:gotz}, item \ref{rem:gotz_ii}). Hence, the evaluation at $s$ of the Hilbert function of $S^{(j,n)}/I$  is the same as the evalutation at $s$ of the Hilbert polynomial of $I$. This ensures that if $J^\sat\subseteq  S^{(l+1,n)}$ has  Hilbert polynomial $\Delta P(z)$, then the Hilbert polynomial of $J^\sat\cdot \Sk$ is $P(z)+q$  for some $q\in \mathbb Z$.
\end{remark}

\section{Borel-fixed ideals}

In the present section, we will recall the definition and properties of \emph{Borel-fixed ideals}. These ideals are interesting by themselves, since they have a rich combinatorial structure, but also they are used to investigate properties of other polynomial ideals. Indeed, if $k$ is infinite, it is possible to compute the  \emph{generic initial ideal} of a polynomial ideal in $\Sk$. The generic initial ideal of an ideal $I$ is Borel-fixed and it is used to investigate properties of $I$ \cite{Green}. 
If $\mathrm{char}(k)=0$, it is quite simple to deal with a Borel-fixed ideal, while if $\mathrm{char}(k)>0$, the combinatorial structure of a Borel-fixed ideal is more entangled. However, we will be able to handle it using results of the previous sections.

\begin{definition}
Let $\mathrm{GL}(n-\ell+1,k)$ be the general linear group, that is the group of invertible $(n-\ell+1)\times(n-\ell+1)$-matrices with entries in $k$. Every $g=(g_{ij})_{i,j\in \{0,\dots, n-\ell\}}\in \mathrm{GL}(n-\ell+1,k)$ induces an automorphism
\[
\begin{array}{rcl}
g:\Sk&\rightarrow & \Sk\\
f(x_\ell,\dots,x_n)&\mapsto & f(\sum_{j=0}^{n-\ell} g_{0j}x_j,\dots, \sum_{j=0}^{n-\ell} g_{n-\ell \, j}x_j)
\end{array}
\]
 For every ideal $I\subseteq \Sk$, we write $g(I)$ for $(g(f(x_\ell,\dots,x_n))\vert  f \in I)$.
\end{definition}
 
We denote by $\mathcal B$ the Borel subgroup of $\mathrm{GL}(n-\ell+1,k)$ consisting of upper  triangular matrices. 
\begin{definition}
Let $I\subset \Sk$ be a homogeneous ideal. We say that $I$ is \emph{Borel-fixed} if for every $g\in \mathcal B$, $g(I)=I$.
\end{definition}

Every Borel-fixed ideal is monomial \cite[Theorem 15.23]{Eis}.  Furthermore, observe that if the monomial ideal $J$ is Borel-fixed, then $J_{\geq m}$ is. Indeed, for every $g\in \mathrm{GL}(n-\ell+1,k)$, for every $f\in \Sk$, $g(f)$ has the same degree as $f$. Hence, for every $m$, if $J$ is Borel-fixed, then $g(J_m)=J_m$.

\begin{definition}
Let $p$ be a prime number, $a$ and $b$ be natural numbers. We say that  $a\prec_p b$ if  and only if   ${b \choose a}\neq 0 \mod p$. We extend this definition for $p=0$ posing that $a\prec_0 b$ if and only if $a\leq b$ in the usual sense.
\end{definition}

\begin{definition}\label{def:mossep}
For every $i<j$, for every $s>0$, we define the \emph{$s$-th increasing move} on the term $x^\alpha\in \T$ as
\[
e_{i,j}^{+(s)}(x^\alpha)=\frac{x_j^sx^\alpha}{x_i^s}=x_\ell^{\alpha_\ell}\cdots x_i^{\alpha_i-s}\cdots x_j^{\alpha_j+s}\cdots x_n^{\alpha_n}.
\]
We say that the increasing move $e_{i,j}^{+(s)}$ on the term $x^\alpha$  is \emph{$p$-admissible} if $e_{i,j}^{+(s)}(x^\alpha)\in \T$ and $s\prec_p \alpha_i$. 
\end{definition}

\begin{definition}
For a fixed prime $p$, we define the following relation on the terms of $\T$: $x^\alpha<_p x^\beta$ if and only if there is a $p$-admissible increasing move $e_{i,j}^{+(s)}$ such that $e_{i,j}^{+(s)}(x^\alpha)=x^\beta$.
The transitive closure of this relation gives a partial order on the set of monomials of any fixed degree, that we will keep on denoting by $<_{p}$.
\end{definition}

\begin{theorem}\cite[Theorem 15.23]{Eis}\label{combinp}
Let $\mathrm{char}(k)=p\geq 0$ and $J\subset \Sk$ be a monomial ideal. $J$ is Borel-fixed if and only if for every $x^\alpha \in B_J$, if $x^\alpha<_{p}x^\beta$, then $x^\beta \in J$.
\end{theorem}

\begin{remark}\label{combinpbase}
If $p=0$, then a Borel-fixed ideal $J$ is \emph{strongly stable}: for every $x^\alpha\in J$, for every $i<j$ such that $x_i$ divides $x^\alpha$, then $e^{+(1)}_{i,j}(x^\alpha)=\displaystyle\frac{x_j x^\alpha}{x_i}\in J$. If $J$ is strongly stable, then it is Borel-fixed, irrespective of the characteristic of the field $k$.
\end{remark}

In what follows, we will say that the ideal $J$ is \emph{$p$-Borel} meaning that the ideal $J$ is Borel-fixed in $\Sk$, with $\mathrm{char}(k)=p$. 

\begin{example}\label{ex:BorelChar}\ 

\begin{enumerate}
\item \label{BorelChar_i}
Consider $J=(x_2^{11},x_2^{10}x_1,x_2^2x_1^9,x_2x_1^{10})\subseteq S^{(0,2)}$ with $x_2>x_1>x_0$.\\
The ideal $J$ is not $0$-Borel:  for instance the term $x_2^2x_1^{9}$ belongs to $J$, but $e^{+(1)}_{1,2}(x_2^2x_1^9)=x_2^3x_1^8$ does not. The ideal $J$ is not $5$-Borel: consider the term $x_2x_1^{10}$, and observe that $e^{+(5)}_{1,2}(x_2x_1^{10})=x_2^6x_1^5$ does not belong to $J$. The ideal $J$ is $3$-Borel: one can check that for every $x^\alpha\in B_J$ the condition of Theorem \ref{combinp} holds.

\item
The ideals $J_1=(x_n^3, x_n^2x_{n-1},x_nx_{n-1}^2)$ and $J_2=(x_n^3,x_n^2x_{n-1},x_n^2x_{n-2})$ in  $S^{(0,n)}$ are Borel-fixed for every characteristic of the field $k$, because they are strongly stable. 
 The ideal $J=(x_n^{p^t},x_{n-1}^{p^t},\cdots, x_{j+1}^{p^t}, x_j^{p^t})\subset S^{(0,n)}$ is $p$-Borel, for any $p>0$.  \cite[page 357]{Eis}
\end{enumerate}
\end{example}

In the following corollary, we just observe that a $p$-Borel ideal  is quasi-stable, hence the Pommaret basis is a suitable set of generators to handle it.
\begin{corollary}\label{PBorelAreQS}
Let $J$ be a Borel-fixed ideal in $\Sk$, with $\mathrm{char}(k)=p\geq 0$. Then $J$ is quasi-stable and has a Pommaret basis $\PP(J)$.
\end{corollary}
\begin{proof}
If $J$ is $p$-Borel, then it fulfills condition (\ref{BorelEquiv_iii}) of Theorem \ref{thm:BorelEquiv}.
\end{proof}

Corollary \ref{PBorelAreQS} cannot be reversed, as shown by the following example.

\begin{example}
Consider the ideal $J=(x_1,x_2^2,x_3)\subset S^{(0,3)}$.
$J$ is quasi-stable but is not $p$-Borel for any value of $p$.\\
Suppose there is $p\geq 0$ such that $J$ is $p$-Borel: in particular, $1\prec_p 1$ for every $p\geq 0$, hence 
$(x_1/\min(x_1))\cdot x_2$ should belong to $J$, but this is not the case.
\end{example}

\section{An algorithm to compute $p$-Borel ideals with a given Hilbert polynomial}\label{sec:AlgoBorel}

From  Algorithm  \ref{alg:QSideals} \textsc{QuasiStable} presented in Section \ref{sec:AlgorithmQS}, we can obtain the Borel-fixed ideals with an assigned Hilbert polynomial for a field with characteristic $p\geq 0$: it is sufficient to run Algorithm  \ref{alg:QSideals}  and then check which of the ideals output are $p$-Borel, using Theorem \ref{combinp}. However, with a modified version of Lemma \ref{CostrPom}, we can modify Algorithm  \ref{alg:QSideals} and directly obtain the set of $p$-Borel ideal with a given Hilbert polynomial. In our tests, for a fixed Hilbert polynomial, this second way to compute $p$-Borel ideals was faster than computing the whole set of quasi-stable ideals (see Example \ref{ex:borelcharp}).

\begin{definition}\label{minBp}
Let $J$ be a $p$-Borel ideal and consider $x^\alpha \in J$. $x^\alpha$ is \emph{$p$-minimal} if there is no other term in $J$ smaller than $x^\alpha$ w.r.t $<_{p}$.
\end{definition}

\begin{lemma}\label{BorelMinP}
Let $J$ be a $p$-Borel ideal in $\Sk$, let $P(z)$ be its Hilbert polynomial and let $r$ be the Gotzmann number of $P(z)$. For every $s\geq r$, let $x^\alpha\in J_s$ be a $p$-minimal and St-minimal term  in $J_s$ with $\min(x^\alpha)=x_l$. 
Let $M$ be the set of terms $(J_s\cap \T)\setminus \lbrace x^\alpha\rbrace$.  Then $(M)$ is $p$-Borel and  has Hilbert polynomial $P(z)+1$. 
\end{lemma}

\begin{proof}
The ideal $J$ is $p$-Borel, hence it is quasi-stable by Corollary \ref{PBorelAreQS}. Hence, Lemma \ref{CostrPom} applies and we obtain that the ideal generated by $M=J_s\cap \T\setminus \lbrace x^\beta\rbrace$ is quasi-stable and has Hilbert polynomial $P(z)+1$, because the term $x^\alpha$ is St-minimal and $\min(x^\alpha)=x_\ell$.
The ideal  generated by $M$ is also $p$-Borel, since the term $x^\alpha$ that we remove from $J_s$ is $p$-minimal.
\end{proof}

\begin{lemma}\label{lem:x1satpB}
Let $I$ and $J$ be $p$-Borel ideals in $\Sk$, let $P_1(z)$ be the Hilbert polynomial of $I$ and $P_2(z)$ be the Hilbert polynomial of $J$. If, for every $s\gg 0$, we have $I_s\subset J_s$ and $P_1(z) = P_2 (z) + a$, with $a\in \mathbb N$, then $I$ and $J$ have the same $x_{\ell+1}$-saturation and for every $s\gg 0$ there is $x^\alpha \in J_s\setminus I_s$, with $\min(x^\alpha)=x_\ell$, which is St-minimal and $p$-minimal.
\end{lemma}
\begin{proof} 
Consider $s\geq \max\lbrace \reg(I),\reg(J)\rbrace$. By Lemma \ref{lem:x1sat}, we know that for every $s\gg 0$ there is at least one term $x^\alpha \in J_s\setminus I_s$, with $\min(x^\alpha)=x_\ell$, which is St-minimal. Suppose by contradiction that among these St-minimal terms there are no $p$-minimal terms. This is against the hypothesis that $I$ (and its truncation at $s\gg 0$) is $p$-Borel.
\end{proof}

In general, the set of  $p$-minimal terms and that of St-minimal terms of a $p$-Borel ideal are not included one in the other, but in the hypothesis of Lemma \ref{lem:x1satpB}, there is always a non-empty intersection between these two sets.
\begin{example}
Consider $S^{(0,3)}$, $\mathrm{char}(k)=2$, and the monomial ideal $J=(x_3^2,x_2^2)$ which is $2$-Borel. The Hilbert polynomial  is $P(z)=4z$, whose Gotzmann number is $r=6$.
According to Definitions \ref{def:StMin} and  \ref{minBp},   the St-minimal terms of $J_{\geq r}$ are $x_0^4x_2^2, x_0^4x_3^2$  while the $2$-minimal elements  of $J_{\geq r}$ are the terms $x_0^4x_2^2, x_0^3x_1x_2^2$.
\end{example}

Thanks to Lemma \ref{BorelMinP}, we can simply obtain an Algorithm to compute all the saturated $p$-Borel ideals in $\Sk$ having Hilbert polynomial $P(z)$, without computing the whole set of quasi-stable ideals with Hilbert polynomial $P(z)$. 
It is sufficient to write a modified version of Algorithms  \ref{alg:Remove} and  \ref{alg:QSideals} using Lemma \ref{BorelMinP} instead of Lemma \ref{CostrPom}. 
More precisely, it is easy to write a procedure $p$\textsc{-MinimalElements}$(J,\ell,n,s,p)$ which computes the terms $x^\alpha$ in $J_s$ with $\min(x^\alpha)=x_\ell$ and $x^\alpha$ both  $p$-minimal and St-minimal for $J_{\geq s}$. 

Then we obtain Algorithm \ref {alg:pRemove} \textsc{$p$-Remove} using  $p$\textsc{-MinimalElements}  instead of \textsc{StMinimal}.

\begin{algorithm}[h] 
\caption{\label{alg:pRemove}$p$\textsc{-Remove}$(p,I,\ell, n, s , q)$}
\begin{algorithmic}[1]
\REQUIRE  $p$ characteristic of the coefficient field;
\REQUIRE  $I$ set of monomials generating a $p$-Borel ideal;
\REQUIRE $\ell$ first index of the variables in the polynomial ring;
\REQUIRE  $n$  least index of the variables of the polynomial ring;
\REQUIRE  $s$ upper bound on $\reg(I)$;
\REQUIRE  $q$ number of terms to remove from $I_s$;
\ENSURE $L$ set of the saturated $p$-Borel ideals $J$  obtained by removing $q$ terms divisible by $x_\ell$ from $I_s$;
\STATE $L\leftarrow \emptyset$;
\IF{$q$=0}
\RETURN\label{algR:step0} $L \cup I^\sat$;
\ELSE
\STATE\label{stepMinp} $M\leftarrow p\textsc{-MinimalElements}(J,\ell,n,s,p)$;
\FORALL{$x^\alpha\in M$}
\STATE\label{stepalgo4ricors} $L\leftarrow L\cup p\textsc{-Remove}(p,I_{s}\setminus\lbrace x^\alpha\rbrace,\ell, n,  s , q-1)$;
\ENDFOR
\ENDIF
\RETURN $L$;
\end{algorithmic}
\end{algorithm}

\begin{theorem}
Algorithm \ref{alg:pRemove}, $p$\textsc{-Remove}$(p,I,\ell, n, s , q)$, returns the set of all $p$-Borel ideals in the polynomial ring $\Sk$ contained in $I_s$, having the same $x_{\ell+1}$-saturation as $I$  
and having Hilbert polynomial $P_1(z)+q$, where $P_1(z)$ is the Hilbert polynomial of $I$.
\end{theorem}

\begin{proof}
We simply follow the lines of the proof of Theorem \ref{thm:corrRem}, replacing the argument on Lemma \ref{CostrPom} by Lemma \ref{BorelMinP}.

If $q=0$, Algorithm \ref{alg:pRemove} terminates at line \ref{algR:step0} and its output is correct. 

If $q>0$, at line \ref{stepMinp} the algorithm computes the set of $p$-minimal and St-minimal terms $x^\alpha$ with $\min(x^\alpha)=x_\ell$. By Lemma \ref{BorelMinP}, the set of terms in $I_s\setminus \{x^\alpha\}$ generates a $p$-Borel ideal with Hilbert polynomial $P_1(z)+1$. 

The algorithm terminates because at each recursive call at line \ref{stepalgo4ricors},  the number of terms to remove decreases.

Furthermore, observe that applying Algorithm \ref{alg:pRemove} on the $p$-Borel ideal $I$ with Hilbert polynomial $P(z)$, we obtain as output ideals having the same $x_{\ell+1}$-saturation as $I$, by Lemma \ref{lem:x1satpB}. 
\end{proof}

Finally, Algorithm \ref{alg:PBorel}  computes the list of saturated $p$-Borel ideals in $\Sk$ with Hilbert polynomial $P(z)$ in the following way:  at line \ref{steprem} of Algorithm  \ref{alg:QSideals} \textsc{QuasiStable} we call Algorithm \ref{alg:pRemove} instead of Algorithm \ref{alg:Remove}.

\begin{algorithm}[h]
\begin{algorithmic}[1]
\caption{\label{alg:PBorel}  \textsc{Borel}$(\ell,n,P(z),s,p)$}
\REQUIRE $\ell$ first index of the variables in the polynomial ring;
\REQUIRE  $n$  least index of the variables of the polynomial ring;
\REQUIRE  $P(z)$ admissible Hilbert polynomial;
\REQUIRE $s$  positive integer upper bounding the Gotzmann number of $P(z)$;
\REQUIRE  $p$ characteristic of the coefficient field;
\ENSURE $F$ set of the saturated $p$-Borel ideals $J$ in the polynomial ring $\Sk$ having Hilbert polynomial $P(z)$;
\IF{$P(z)=0$}
\RETURN $\lbrace (1)\rbrace$; 
\ELSE
\STATE$E\leftarrow \textsc{Borel}(\ell+1,n,\Delta P(z),s,p)$; 
\STATE $F\leftarrow \emptyset$;
\FORALL{$J\in E$}
\STATE $I\leftarrow J\cdot k[x_{\ell},\dots,x_n]$;
\STATE $q\leftarrow P(s)-{n-\ell+s \choose s}+\dim_k I_s$;
\IF{$q\geq 0$}
\STATE {$F\leftarrow F\cup p\textsc{-Remove}(p,I,\ell,n,s,q)$};
\ENDIF
\ENDFOR
\ENDIF
\RETURN $F$;
\end{algorithmic}
\end{algorithm}

\begin{theorem}
Algorithm \ref{alg:PBorel},  \textsc{Borel}$(\ell,n,P(z),s,p)$, returns the set of all $p$-Borel sa\-tu\-ra\-ted ideals in the polynomial ring $\Sk$ with Hilbert polynomial $P(z)$.
\end{theorem}

\begin{proof}
It is sufficient to repeat the arguments in the proof of Theorem \ref{thm:algoqs}.
\end{proof}

Observe that if we put $p=0$ as argument of Algorithm \ref{alg:PBorel},   then we obtain the set of stronlgy stable ideals in $\Sk$ having Hilbert polynomial $P(z)$, that is we obtain exactly the same algorithm as the one presented in \cite{CLMR11}. An improved version of the latter algorithm is presented in \cite{L}: on the one hand the algorithm corresponding to $p\textsc{-Remove}$, with an extra input argument, avoids to compute twice the same ideal; on the other hand the efficiency of the procedure is improved by  a slim structure to store the data and quick implementations of the basic operations (see \cite[Sections 4.1 and 4.2]{L}).

\section{Examples}\label{sec:examples}

We  implemented  prototypes of  Algorithms \ref{alg:Remove}, \ref{alg:QSideals}, \ref{alg:pRemove}, \ref{alg:PBorel} and also of  the au\-xiliary functions
\textsc{GotzmannNumber}, \textsc{St-Minimal}, $p$\textsc{-MinimalElements},   for Ma\-ple 16 \cite{Maple}. These implementations can be largely improved, in particular for Algorithms \ref{alg:Remove} and  \ref{alg:pRemove}, for instance following the lines of \cite{L}. Hence, we will not list timings of computation, but we highlight that these prototypes, although non-optimal, allowed us to explicitely compute a wide range of examples of quasi-stable ideals and $p$-Borel ideals with a given Hilbert polynomial.

These prototypes, the following examples and many others are available at the webpage {\footnotesize\tt https://sites.google.com/site/cristinabertone}

\medskip

From the several computations performed, we have practical evidence that in order to compute the list of $p$-Borel ideals with Hilbert polynomial $P(z)$ for a fixed $p\geq 0$, it is faster to use the Algorithm \ref{alg:PBorel} $\textsc{Borel}$, which relies on Lemma \ref{BorelMinP}, than Algorithm  \ref{alg:QSideals} \textsc{QuasiStable} and then detect $p$-Borel ideals, by checking for every ideal in the output of  Algorithm  \textsc{QuasiStable} the equivalent condition of Theorem \ref{combinp}.

\begin{example}\label{ex:borelcharp}
We consider $P(z)=6z-3$ whose Gotzmann number is $r=12$ and look for $p$-Borel ideals in $S^{(0,3)}$.

First, we run  \textsc{QuasiStable}$(0,3,P(z),r)$, and we obtain 322 quasi-stable saturated monomial ideals. From this list, we can detect, for every $p\geq 0$, which ideals are $p$-Borel.

If we directly compute saturated $0$-Borel ideals by Algorithm $\textsc{Borel}$, we find 31 ideals. If we compute the saturated  $2$ and $3$-Borel ideals with Hilbert polynomial $P(z)$ by using  $\textsc{Borel}$, we get 35 and 34 ideals, respectively. In these three cases, the time of computation running Algorithm $\textsc{Borel}$ is  far lower than that of running \textsc{QuasiStable}$(0,3,P(z),r)$: the computation of the quasi-stable ideals took about 8 times the time needed for the computation of $p$-Borel ideals for $p\in\lbrace 0,2,3\rbrace$.
\end{example}

Also in the simplest case, a constant Hilbert polynomial in $S^{(0,2)}$, for different values of the characteristic of the field, we obtain different sets of saturated $p$-Borel ideals with Hilbert polynomial $P(z)$. Remember that every $0$-Borel ideal is $p$-Borel for every prime $p$ (Remark \ref{combinpbase}).

\begin{example}
We consider $S^{(0,2)}$ and the Hilbert polynomial $P(z)=14$, whose Gotzmann number is $r=14$. We directly compute the set $p$-Borel ideals with Hilbert polynomial $P(z)$ in $S^{(0,2)}$, for $p\in \{0,2,3,5,7\}$ using Algorithm $\textsc{Borel}$.

For $p=0$, we obtain 22 saturated  $0$-Borel ideals (namely, strongly stable ideals) having Hilbert polynomial $P(z)$; they are
\[
\begin{array}{l}
J_1=(x_{2}, x_{1}^{14}), \quad J_2=(x_{2}^{2}, x_{2}x_{1}^{6},x_{1}^{8}), \quad J_3=(x_{2}^{2}, x_{1}^{5}x_{2}, x_{1}^{9}), \\
J_4=(x_{2}^{2}, x_{1}^{4}x_{2}, x_{1}^{10}), \quad J_5=(x_{2}^{2}, x_{1}^{3}x_{2}, x_{1}^{11}),\quad J_6=(x_{2}^{2}, x_{1}^{2}x_{2}, x_{1}^{12}),\\
J_7=(x_{2}^{2}, x_{1}x_{2}, x_{1}^{13}), \quad J_8=(x_{2}^{3}, x_{1}^{3}x_{2}^{2}, x_{1}^{5}x_{2}, x_{1}^{6}), 
J_9=(x_{2}^{3}, x_{1}x_{2}^{2}, x_{2}x_{1}^{6}, x_{1}^{7}), \\
 J_{10}=(x_{2}^{3}, x_{1}^{2}x_{2}^{2}, x_{1}^{5}x_{2}, x_{1}^{7}),\quad 
\quad J_{11}=(x_{2}^{3}, x_{1}^{3}x_{2}^{2}, x_{1}^{4}x_{2}, x_{1}^{7}),\\
J_{12}=(x_{2}^{3}, x_{1}x_{2}^{2}, x_{1}^{5}x_{2}, x_{1}^{8}), \quad J_{13}=(x_{2}^{3}, x_{1}^{2}x_{2}^{2}, x_{1}^{4}x_{2}, x_{1}^{8}), \\
J_{14}=(x_{2}^{3}, x_{1}x_{2}^{2}, x_{1}^{4}x_{2}, x_{1}^{9}), \quad
J_{15}=(x_{2}^{3}, x_{1}^{2}x_{2}^{2}, x_{1}^{3}x_{2}, x_{1}^{9}), \\ J_{16}=(x_{2}^{3}, x_{1}x_{2}^{2}, x_{1}^{3}x_{2}, x_{1}^{10}),
\quad J_{17}=(x_{2}^{3}, x_{1}x_{2}^{2}, x_{1}^{2}x_{2}, x_{1}^{11}),\\
J_{18}=(x_{2}^{4}, x_{1}^{2}x_{2}^{3}, x_{1}^{3}x_{2}^{2}, x_{1}^{4}x_{2}, x_{1}^{5}), \quad J_{19}=(x_{2}^{4}, x_{1}x_{2}^{3}, x_{1}^{2}x_{2}^{2}, x_{1}^{5}x_{2},x_{1}^{6}), \\ 
J_{20}=(x_{2}^{4}, x_{1}x_{2}^{3}, x_{1}^{3}x_{2}^{2}, x_{1}^{4}x_{2}, x_{1}^{6}),\quad 
J_{21}=(x_{2}^{4}, x_{1}x_{2}^{3}, x_{1}^{2}x_{2}^{2}, x_{1}^{4}x_{2}, x_{1}^{7}), \\ J_{22}=(x_{2}^{4}, x_{1}x_{2}^{3}, x_{1}^{2}x_{2}^{2}, x_{1}^{3}x_{2},x_{1}^{8}).
\end{array}\]

Every one of the $0$-Borel ideals with Hilbert polynomial $P(z)$ is also $p$-Borel, for every prime $p$. Hence, for the other values of $p$ considered, here we just list the saturated  $p$-Borel ideals which are not strongly stable.

$\begin{array}{rl}
p=2: &(x_2^4, x_2^3x_1^2, x_1^4), (x_2^3,x_2^2 x_1^2, x_1^6), (x_2^3, x_2x_1^4, x_1^6), (x_2^3, x_2x_1^2, x_1^{10}),\\& (x_2^4, x_2x_1^4,x_2^3x_1,  x_1^5), (x_2^4, x_2^2x_1^2, x_2x_1^4, x_1^6);\\
p=3: &(x_2^3, x_2^2x_1^2,x_1^6), (x_2^3, x_2x_1^3,x_1^8), (x_2^4, x_2^3x_1^2, x_2x_1^3,x_1^6),\\& (x_2^4, x_2^3x_1, x_2x_1^3,x_1^7);\\
p=5: &(x_2^3, x_2^2x_1^4,x_1^5), (x_2^4, x_2^3x_1, x_2^2x_1^3,x_1^5);\\
p=7: &(x_2^2,x_1^7).
\end{array}$
\end{example}


\section*{Acknowledgements}
The author deeply thanks prof. Margherita Roggero for suggesting the topic of Borel-fixed ideals in positive characteristic and her encouragement, and prof. Werner Seiler for valuable discussions and suggestions on quasi-stable ideals and Pommaret bases.

\bibliographystyle{amsplain}

\begin{thebibliography}{10}

\bibitem{Ba1982}
David Bayer, \emph{The division algorithm and the {H}ilbert scheme}, Ph.D.
  thesis, Harvard U., Cambridge, 1982.

\bibitem{BG06}
Isabel Bermejo and Philippe Gimenez, \emph{Saturation and
  {C}astelnuovo-{M}umford regularity}, J. Algebra \textbf{303} (2006), no.~2,
  592--617.

\bibitem{BCLR}
Cristina Bertone, Francesca Cioffi, Paolo Lella, and Margherita Roggero,
  \emph{Upgraded methods for the effective computation of marked schemes on a
  strongly stable ideal}, J. Symbolic Comput. \textbf{50} (2013), 263--290.

\bibitem{BCR}
Cristina Bertone, Francesca Cioffi, and Margherita Roggero, \emph{A division
  algorithm in an affine framework for flat families covering {H}ilbert schemes},
  available at arXiv:1211.7264 [math.AC], 2013.

\bibitem{BLR}
Cristina Bertone, Paolo Lella, and Margherita Roggero, \emph{A {B}orel open
  cover of the {H}ilbert scheme}, J. Symbolic Comput. \textbf{53} (2013),
  119--135.

\bibitem{BH}
Winfried Bruns and J{\"u}rgen Herzog, \emph{Cohen-{M}acaulay rings}, Cambridge
  Studies in Advanced Mathematics, vol.~39, Cambridge University Press,
  Cambridge, 1993.

\bibitem{CS05}
Giulio Caviglia and Enrico Sbarra, \emph{Characteristic-free bounds for the
  {C}astelnuovo-{M}umford regularity}, Compos. Math. \textbf{141} (2005),
  no.~6, 1365--1373.

\bibitem{CMR13}
Michela Ceria, Teo Mora, and Margherita Roggero, \emph{Term-ordering free
  involutive bases},  {J}.  {S}ymbolic {C}omput., \textbf{68} (2015), part 2, 87--108.

\bibitem{CLMR11}
Francesca Cioffi, Paolo Lella, Maria~Grazia Marinari, and Margherita Roggero,
  \emph{Segments and {H}ilbert schemes of points}, Discrete Math. \textbf{311}
  (2011), no.~20, 2238--2252.

\bibitem{CR}
Francesca Cioffi and Margherita Roggero, \emph{Flat families by strongly stable
  ideals and a generalization of {G}r\"obner bases}, J. Symbolic Comput.
  \textbf{46} (2011), no.~9, 1070--1084.

\bibitem{Eis}
David Eisenbud, \emph{Commutative algebra  with a view toward algebraic geometry}, Graduate Texts in Mathematics, vol.
  150, Springer-Verlag, New York, 1995.
  
  \bibitem{Eis2}
David Eisenbud, \emph{The geometry of syzygies. A second course in commutative algebra and algebraic geometry}, Graduate Texts in Mathematics, vol.
  229, Springer-Verlag, New York, 2005.
  
\bibitem{FR}
Gunnar Floystad and Margherita Roggero, \emph{Borel degenerations of
  arithmetically {C}ohen-{M}acaulay curves in ${P}^3$}, Internat. J. Algebra
  Comput. \textbf{24} (2014), no.~05, 715--739.

\bibitem{GB}
Vladimir~P. Gerdt and Yuri~A. Blinkov, \emph{Involutive bases of polynomial
  ideals}, Math. Comput. Simulation \textbf{45} (1998), no.~5-6, 519--541,
  Simplification of systems of algebraic and differential equations with
  applications.

\bibitem{Green}
Mark~L. Green, \emph{Generic initial ideals}, Six lectures on commutative
  algebra ({B}ellaterra, 1996), Progr. Math., vol. 166, Birkh\"auser, Basel,
  1998, pp.~119--186.
%

\bibitem{HPV03}
J{\"u}rgen Herzog, Dorin Popescu,  and Marius Vladoiu, \emph{On the {E}xt-modules of ideals of {B}orel type}, {Commutative algebra ({G}renoble/{L}yon, 2001)}, Contemp. Math. \textbf{331} (2003),  {171--186}.


\bibitem{L}
Paolo Lella, \emph{An efficient implementation of the algorithm computing the
  borel-fixed points of a hilbert scheme}, Proceedings of the International
  Symposium on Symbolic and Algebraic Computation, ISSAC, 2012, pp.~242--248.

\bibitem{Mall1998}
Daniel Mall, \emph{On the relation between {G}r\"obner and {P}ommaret bases},
  Appl. Algebra Engrg. Comm. Comput. \textbf{9} (1998), no.~2, 117--123.

\bibitem{Maple}
Maplesoft, \emph{Maple - technical computing software for engineers,
  mathematicians, scientists, instructors and students}, 2012.

\bibitem{MN}
Dennis Moore and Uwe Nagel, \emph{Algorithms for strongly stable ideals}, Math.
  Comp. \textbf{83} (2014), no.~289, 2527--2552.

\bibitem{Seiler2009I}
Werner~M. Seiler, \emph{A combinatorial approach to involution and
  {$\delta$}-regularity. {I}. {I}nvolutive bases in polynomial algebras of
  solvable type}, Appl. Algebra Engrg. Comm. Comput. \textbf{20} (2009),
  no.~3-4, 207--259.

\bibitem{Seiler2009II}
Werner~M. Seiler, \emph{A combinatorial approach to involution and
  {$\delta$}-regularity. {II}. {S}tructure analysis of polynomial modules with
  {P}ommaret bases}, Appl. Algebra Engrg. Comm. Comput. \textbf{20} (2009),
  no.~3-4, 261--338.

\bibitem{vas}
Wolmer~V. Vasconcelos, \emph{Computational methods in commutative algebra and
  algebraic geometry}, Algorithms and Computation in Mathematics, vol.~2,
  Springer-Verlag, Berlin, 1998, With chapters by David Eisenbud, Daniel R.
  Grayson, J{\"u}rgen Herzog and Michael Stillman.

\end{thebibliography}

\end{document}